\documentclass[11pt]{article}
\usepackage{times}
\usepackage{graphicx}
\usepackage{amsmath,amsthm,amssymb}
\usepackage{epstopdf}
\usepackage[height=8.1in,width=6.0in]{geometry}

\newtheorem{theorem}{Theorem}

\theoremstyle{definition}

\newtheorem*{definition*}{Definition}
\newtheorem*{example*}{Example}
\newtheorem*{remark*}{Remark}
\newtheorem*{question*}{Question}
\newtheorem*{problem*}{Problem}
\newtheorem*{note*}{Note}
\newtheorem*{claim*}{Claim}
\newtheorem*{fedex*}{FedEx Problem}
\newtheorem*{altfedex*}{Alternative FedEx Problem}
\newtheorem*{fredex*}{The FredEx Problem}

\newcommand{\R}{\mathbb{R}}

\newcommand{\ra}{\rightarrow}

\newcommand{\xset}{x_1,x_2,\ldots,x_k}

\title{\textbf{The FedEx Problem} \thanks{\emph{College Math. J.} \textbf{41} (2010), no. 3,  222--232}}
\author{Kent E. Morrison \thanks
    {American Institute of Mathematics,
    360 Portage Avenue, 
    Palo Alto, CA 94306}   
}
\date{}
\begin{document}

\maketitle

\large
\renewcommand{\baselinestretch}{1.2}   
\normalsize

In April of 1973 a small company, Federal Express, began package delivery operations by flying 186 packages to 25 cities in the U.S. from Memphis. In order to deliver small packages quickly the company used a novel strategy in the shipping business. Every day packages picked up from all over the country were flown to Memphis where they were sorted and then flown to their destinations for delivery the next day. The strategy has been extremely successful and has been adopted by other shippers such as UPS, which established a shipping hub in Louisville. Meanwhile, Federal Express has grown into a large company, now known as FedEx, with worldwide operations.  

According to the company website \cite{FedEx09},  Memphis was 
\begin{quote}
selected for its geographical center to the original target market cities for small packages. In addition, the Memphis weather was excellent and rarely caused closures at Memphis International Airport. The airport was also willing to make the necessary improvements for the operation and had additional hangar space readily available.
\end{quote}

It is the question of the ``geographical center'' that is the focus of this article. Suppose that FedEx were choosing its hub today with the assumption that anyone in the U.S. is equally likely to ship a package to anyone else in the country. The goal is to minimize the average distance the package must be shipped. Since we are assuming that the sender's and receiver's locations are independent and identically distributed, the average distance that a package travels is twice the average distance from the sender to the hub location, and so the optimal location is a point located with minimal average distance to the population. This is equivalent to a point that minimizes the sum of all the distances to the members of the population. Such a point we will call a \emph{hub} for the population. By the term \emph{FedEx problem}, we mean the related questions of the existence, uniqueness, and determination of hubs.

In the first section we deal with the FedEx problem for a population in $\R^n$ using the Euclidean distance, and in section 2 we actually find the hub for the U.S. population based on the data from the 2000 census and with the distance between points measured along great circles. In the last section we consider the FedEx problem in more general metric spaces and with more probability distributions.

\section*{The Euclidean FedEx Problem}

The Euclidean FedEx problem in one dimension has a well-known solution: a hub for the population is any median.  Here is a quick proof. For a population located at the points $x_1\leq x_2 \leq \cdots \leq x_k$, a hub must be between $x_1$ and $x_k$, for otherwise all the distances to the $x_i$ could be decreased by moving the potential hub toward the points. As far as $x_1$ and $x_k$ are concerned, any point between them is equally optimal because the sum of the distances to those two points is constant. Thus, we can eliminate $x_1$ and $x_k$ from the population and reduce the problem to a smaller population. We continue to eliminate the endpoints until there is either one or two left. If there is one point left (corresponding to $k$ odd), then that point is the unique hub, and if there are two points left (corresponding to $k$ even), then any point between them is a hub. The results in either case are medians for the population locations. 

In higher dimensions the story starts with Fermat, who generally receives credit for first posing the following problem, which we have updated and renamed for the twenty-first century.
 \begin{fredex*}
Fred has been married and divorced three times. He has three children, one from each marriage, and the children live with Fred's ex-wives. Each weekend Fred visits one of his children in a regular rotation. Where should he live in order to minimize the distance he travels?
Let the locations of the residences be $a, b, c \in \R^2$. We are assuming that they are near enough so that we do not have to take into account the curvature of the Earth and that the metric is Euclidean distance. The mathematical problem is to find the point $h$ in $\R^2$ that minimizes
 \[ f(x)=  \Vert x-a\Vert + \Vert x-b\Vert + \Vert x-c\Vert. \] 
\end{fredex*}
Torricelli gave the first solution to Fermat's problem and eventually offered several different proofs. 
The solution, to be discussed in more detail following the proof of Theorem 1, has two cases. First, if the angles of the triangle $abc$ are less than $2\pi/3$, then $h$ is the point within the triangle such that the lines from $h$ to the vertices form three equal angles of size $2 \pi/3$. Second, if some angle is larger than $2 \pi/3$, then $h$ is the vertex of that angle. Note that the solution can be constructed by compass and straightedge.

\begin{figure}
\centerline {
\includegraphics[width=5in]{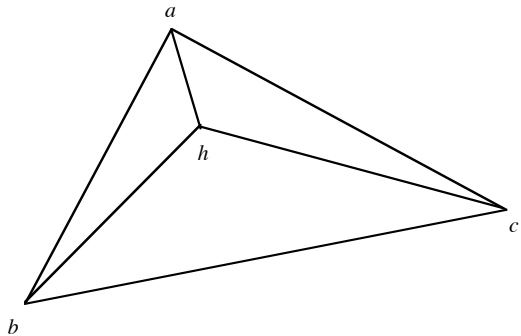}  
}
\vspace{-30mm}
\caption{Solution of the FredEx problem with $h$ inside the triangle.}
\end{figure}

Since then the problem of minimizing total or average distance has arisen repeatedly and in different contexts so that several different names are attached. In addition to being called the Fermat-Torricelli problem, it is known as the Weber problem---named for the economist Alfred Weber who was interested in the problem in connection with the location of industries. In the sub-area of operations research known as location science, the problem is called the \emph{median problem} or the \emph{single facility location problem}. Statisticians may refer to the minimal point as the spatial, multivariate, or multidimensional median. For a survey of the problem, its generalizations, and its history see the article by 
Wesolowsky \cite{ Wesolowsky93} with its extensive bibliography or the papers in the collection edited by Drezner and Hamacher \cite{DreznerHamacher02}.

For the general FedEx problem in $n$-dimensional Euclidean space, we consider a population of size $k$ located at the points $x_1,x_2,\ldots,x_k \in \R^n$, which are not necessarily distinct. The function to be minimized is the sum of the distances from the $x_i$ to a variable point $x \in \R^n$
\[ f(x) = \sum_i \|x-x_i\|. \]
Just as in one dimension where the hub is located between the extremes, a hub in higher dimensions should be located ``between'' the points, which means, as we will show in Theorem 1, that the hub lies in the convex hull of the points (i.e., the smallest convex set containing them). Recall that a subset of $\R^n$ is convex if, for any two points in the set, the points on the line segment between them are also in the set. The convex hull of $\xset$ can be seen as the image of the compact set $\{(\alpha_1,\ldots,\alpha_k) | \alpha_i \geq 0, \, \sum \alpha_i =1\}$ under the continuous map sending $(\alpha_1,\ldots,\alpha_k)$ to  $\sum \alpha_i x_i$. In particular, the convex hull is compact.

The function $f: \R^n \ra \R$ is continuous everywhere and differentiable at all $x \in \R^n$ except the points $x_i$. At a point $x$ not equal to any of the $x_i$ the derivative $Df(x)$ is the linear map from $\R^n$ to $\R$ defined by 
\begin{equation} \label{Df} Df(x)(v)=
\sum_i \left\langle \frac{x-x_i}{\|x-x_i\|}, v \right\rangle 
\end{equation}

We now have the ingredients to prove the existence and uniqueness of a hub for a finite set of non-collinear points in $\R^n$. Although this result has probably been proved again and again, we cannot find it in the literature clearly stated with a complete proof. A brief article by Haldane \cite{Haldane48} contains the result for $\R^2$ without mention of the convex hull.
\begin{theorem}[Existence and Uniqueness of a Hub] \label{exuniq}
For any non-collinear points $x_1,x_2,\ldots,x_k$ in $\R^n$ there is a unique hub contained in the convex hull of the points.
\end{theorem} 
\begin{proof} 
Let $Z$ be the convex hull of $\xset$ and let $x$ be a point not in $Z$. There is a separating hyperplane $H$ between $x$ and $Z$. (See, for example, the Basic Separation Theorem in \cite[p. 158]{PSU93}). Let $n$ be the unit normal to $H$ pointing toward $x$. Thus, $\langle x,n \rangle > \langle z,n \rangle$  for all $z \in Z$. In particular, $\langle x,n \rangle > \langle x_i,n \rangle $ for $i=1,\ldots,n$. Then it follows that $Df(x)(n) >0$, and so $f$ decreases from $x$ in the direction $-n$. More precisely, for some $\epsilon > 0$, $f(x-\epsilon n) < f(x)$. Therefore, a minimum of $f | Z$, which exists because $f$ is continuous and $Z$ is compact, will actually be a global minimum of $f$.

Next we show that the function $f$ is strictly convex. This means that for all $x, y \in \R^n$ and any $t$ in the open interval $(0,1)$,
\[  f(t x + (1-t) y) < t f(x) + (1-t)f(y) .\] We have
\begin{align*}
  f(t x + (1-t) y) &= \sum_i \| tx + (1-t)y-x_i \|  \\
    & =\sum_i \| tx -tx_i +(1-t)y - (1-t)x_i \| \\
     &= \sum_i \| t(x-x_i) +(1-t)(y-x_i) \|   \\
    & \leq \sum_i( \|t(x-x_i) \| +\|(1-t)(y-x_i) \|) \\
     &=t \sum_i \|x-x_i \| + (1-t)\sum_i \|y-x_i \|  \\
     &= t f(x) + (1-t)f(y).
\end{align*}
The triangle inequality in the fourth line will be strict, unless $t(x-x_i)$ and $(1-t)(y-x_i)$ are linearly dependent, which is equivalent to $x,y$ and $x_i$ being collinear. Therefore, if the $x_i$ are not collinear, then there are no $x$ and $y$ simultaneously collinear with all the $x_i$, and it follows that the inequality is strict. Thus, $f$ is strictly convex.

Finally, if there were two distinct minima for $f$, say $x$ and $y$, then $f(x)=f(y)$ and \[f(tx +(1-t)y) < tf(x)+(1-t)f(y) = f(x)\] for all $t \in (0,1)$, contradicting the fact that $x$ is a minimum.  \end{proof}

The hub $h$ of the points $\xset$ must be a critical point of $f$, so that either $h=x_i$ for some $i$, or $Df(h)=0$, which means
 \[ \sum_i\frac{h-x_i}{\|h-x_i\|} = 0. \]
 
Now we return to the case of three points in $\R^2$, which we can assume are not collinear. If the hub is one of the points, then it must be the vertex opposite the longest side of the triangle formed by the three points, because $f(x_i)$ is the sum of the lengths of the two sides meeting at $x_i$, and this is minimal when the two sides are the shorter ones. If the hub is not one of the $x_i$, then it lies within the triangle. Let $u_i$ be the unit vector pointing from $h$ to $x_i$. Then $u_1+u_2+u_3=0$. By rotating the coordinate system we may assume that $u_1=(1,0)$. Then $u_2=(a,b)$ and $u_3=(-1-a,-b)$ with $\|u_2 \|^2=a^2+b^2=1$ and $\|u_3\|^2 = (-1-a)^2+b^2=1$. The last two equations imply that $a=-1/2$ and $b= \pm \sqrt{3}/2$. Thus, the angle between any pair of the vectors $u_1,u_2,u_3$ is $2\pi/3$. 

If one of the angles of the triangle is $2\pi/3$ or greater, then at any other point of the closed triangle, the vectors from it to the other two vertices span an angle greater than $2\pi/3$. Such a point cannot be a hub, and so the hub must be the vertex of the large angle (see Figure 2). 
\begin{figure}
\centerline {
\includegraphics[width=5in]{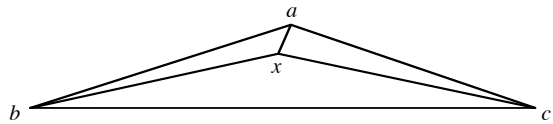}  
}
\vspace{-4cm}
\caption{The solution of the FredEx problem is $h=a$. An interior point $x$ cannot be the hub because $\measuredangle bxc > \measuredangle bac > 2\pi/3.$ }
\end{figure}

If, on the other hand, all the angles are less than $2\pi/3$, then the following continuity argument, illustrated in Figure 3, shows that the hub is an interior point.Let $L_1$, $L_2$, and $L_3$ be three rays issuing from a point $P$ and making equal angles of $2\pi/3$ between them. Place the triangle with one vertex at $P$, one vertex on $L_1$, and the third vertex in the region between $L_1$ and $L_3$. Now continuously slide the triangle so that the vertex originally at $P$ moves out along $L_2$, and the vertex on $L_1$ stays on $L_1$ and moves toward $P$. When the third vertex crosses $L_3$  the vertices of the triangle are on the three rays, and so the point $P$ satisfies the property that the unit vectors from $P$ toward the vertices of the triangle sum to $0$. Therefore,  $P$ is the hub and is an interior point of the triangle. For more on this classical problem we recommend  \cite{Honsberger73} for an extended discussion and \cite{Hajja94} for another proof using calculus. 

\begin{figure}
\centerline{
\includegraphics[width=5in]{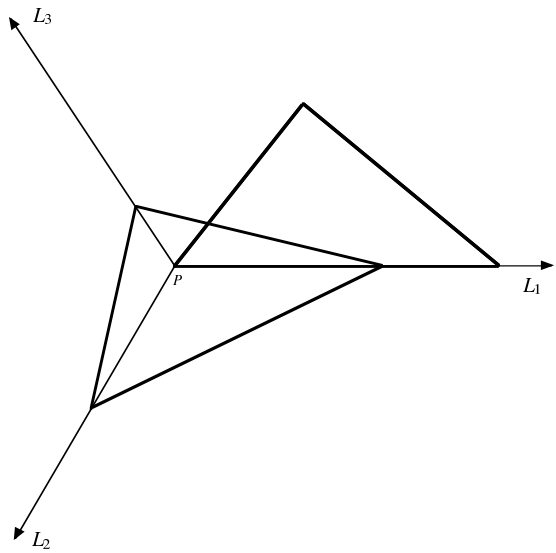}  
}
\vspace{-2cm}
\caption{Beginning with one side on $L_1$ the triangle is moved until the third vertex meets $L_3$. The hub is at $P$.}
\end{figure}

There is also a nice description of the hub for four distinct points in $\R^2$. We assume that the four points are not on a line for otherwise we are in the one-dimensional setting. Again there are two cases; the convex hull is either a quadrilateral or a triangle. In the first case the hub is the intersection of the two diagonals. In the second case one of the four points is in the triangular convex hull of the other three, and that point is the hub. 

To prove this result, first consider the case in which the four points are the vertices of a convex quadrilateral and labeled so that the sequence $x_1,x_2,x_3,x_4$ makes a circuit of the quadrilateral. Let $h$ be the intersection of the diagonals; one diagonal is the line between $x_1$ and $x_3$ and the other is the line between $x_2$ and $x_4$. Let $u_i$ be the unit vector pointing from $h$ to $x_i$. Therefore, $u_1=-u_3$ and $u_2=-u_4$, and so $u_1 +u_2 +u_3+u_4=0$, from which it follows that $h$ is the hub. 

For the second case suppose that $x_4$ is in the triangular convex hull of the other three points. Let $h$ be the hub. Then either $Df(h)=0$ or $h$ is one of the $x_i$. Suppose $Df(h)=0$. Then the unit vectors $u_i$ pointing from $h$ toward the $x_i$ satisfy $u_1 +u_2 +u_3+u_4=0$. We may assume that they are ordered so that they run counter-clockwise around the unit circle. Therefore the points $0, u_1, u_1+u_2, u_1+u_2+u_3$ form a quadrilateral with sides of equal length, and such a quadrilateral must be a rhombus. Thus, $u_1=-u_3$ and $u_2=-u_4$. Then $x_1$ and $x_3$ are on the line through $h$ with direction vector $u_1$, and $x_2$ and $x_4$ are on the line through $h$ with direction vector $u_2$. This means that the $x_i$ are actually the vertices of a convex quadrilateral with $h$ as the intersection of the diagonals. This contradicts our assumption that the convex hull is a triangle. It follows that $h$ must be one of the $x_i$. Finally, we claim that if $x_4 \neq x_i$, then $f(x_4) < f(x_i)$. Consider $i=1$, since the other two are similar. The inequality $f(x_4) < f(x_1)$ is equivalent to
\[  \|x_4-x_2\| +\|x_4-x_3\| < \| x_1-x_2 \| + \|x_1-x_3\| ,\]
which is clear from the picture in Figure 4 (proof left to the reader). Therefore, the hub is $x_4$, the point that is in the convex hull of the other three points. Note that $x_4$ may be on the boundary of the triangle.

\begin{figure}
\vspace{-2cm}
\centerline {
\includegraphics[width=6in]{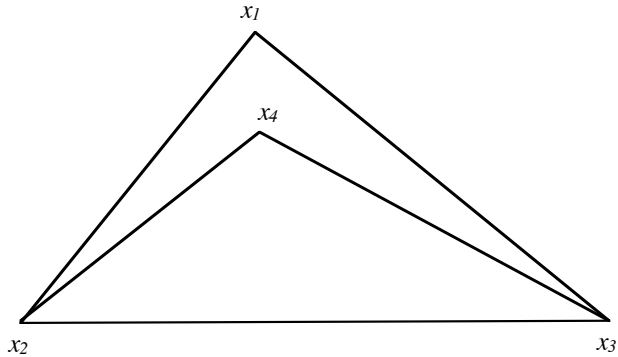}  
}
\vspace{-4cm}
\caption{Proof that $f(x_4)<f(x_1)$.}
\end{figure}

The ``Varignon frame'' is a mechanical device invented by Pierre Varignon (1654--1722) for finding the hub of points $x_1,\ldots,x_k$ in $\R^2$. On a flat piece of wood mark the locations of the points and drill a hole at each point. For each hole take a piece of string and attach a weight to one end--all weights the same. Put the string through the hole with the weight below the board and tie all of the loose ends  together. Hold the board level and above the ground so that the weights can hang freely. The knot tying all the strings together will move to the location of the hub. 

\section*{The U.S. Population Hub}
 
Let's return to the original FedEx Problem of where to establish a shipping hub. Every ten years the U.S. Census Bureau calculates a point called the ``center of population.'' Could this point be the hub we are looking for? According to the current Bureau website \cite{Census01} the center of population is

\begin{quote}
the point at which an imaginary, flat, weightless, and rigid map of the United States would balance perfectly if weights of identical value were placed on it so that each weight represented the location of one person on the date of the census.
\end{quote}

This clearly is not what we want in a hub. The U.S. is large enough so that the curvature matters in measuring distance, and even if we could treat the U.S. area as flat, this population center is the center of gravity for the population distribution, and, thus in effect it is the point minimizing the average squared distances to the population, or, equivalently, the point minimizing the aggregate squared distance to all the people in the country. However, it is an often held misconception that the center of gravity or centroid minimizes the average distance. Even the Census Bureau suffered from this confusion as witnessed by this passage from its Bulletin of the 1920 census (quoted in \cite[p.34]{Eells30}):

\begin{quote}
If all the people in the United States were to be assembled at one place, the center of population would be the point which they could reach with the minimum aggregate travel, assuming that they all travelled in direct lines from their residence to the meeting place.
\end{quote}
After correspondence in 1926 between the Census Bureau, the American Statistical Association, and a group of interested people, the Census Bureau fixed the problem by taking out references to the minimum aggregate travel \cite{Editor30}.

Although the population center of the Census Bureau is not the hub we are looking for, it is interesting nevertheless to understand how it is calculated. This description comes from \cite{Census01} and is the method used for the censuses from 1950 to the most recent one in 2000.
The center of population is given as a pair of numbers $(\bar{\phi},\bar{\lambda})$ representing the center's latitude and the longitude.
  The latitude $\bar{\phi}$ is simply the average latitude of the population:
\[ \bar{\phi} =  \frac{1}{k} \sum_i \phi_i . \]
However, the longitude $\bar{\lambda}$ is \emph{not} the average longitude. Instead it is defined by
 \[  \bar{\lambda}=\frac{\sum_i\lambda_i \cos \phi_i}{\sum_i \cos \phi_i} . \] 
To make some sense of this, notice that the distance from the point with latitude $\phi$ and longitude $\lambda$ to the Greenwich meridian (longitude zero) along the latitude line is $\lambda \cos \phi$. Therefore, the distance to the Greenwich meridian (along lines of constant latitude) averaged over the entire population is $(1/k) \sum_i \lambda_i \cos \phi_i$. This needs to be converted to a longitude value. If we use the average cosine of the latitude of the population in order to convert, then we get the formula of the Census Bureau: 
\[\bar{\lambda} = \frac{\frac{1}{k}\sum_i \lambda_i \cos \phi_i }{\frac{1}{k}\sum_i \cos \phi_i}=
\frac{\sum_i\lambda_i \cos \phi_i}{\sum_i \cos \phi_i} . \] 

The calculation of $\bar{\phi}$ and $\bar{\lambda}$  gives $(\bar{\phi},\bar{\lambda})=(37.7^\circ,91.8^\circ)$, which is a point in Phelps County, Missouri, but there are some difficulties with this definition of the population center. Here is one example. Suppose the population in question consisted of two individuals, one located at $(\phi_1,\lambda_1)=(35^\circ,120^\circ)$, which is near San Luis Obispo, California, and one located at $(\phi_2,\lambda_2)=(35^\circ,80^\circ)$ near Charlotte, North Carolina. It is easy to check that $\bar{\phi}=35^\circ$ and $\bar{\lambda}=100^\circ$, which is the point midway between on the same line of latitude. However, it is difficult to imagine any reasonable definition of population center for this case that would give a point different from the midpoint of the great circle between the two locations, which is $(36.7^\circ, 100^\circ)$. In this case there is quite a discrepancy between the two answers. Both points are in Oklahoma, but they are 116 miles apart.

If the Census Bureau insists that the population center be the center of gravity, then it could be done more accurately by treating the area of the United States as a region on a spherical shell with a unit of mass for each person.
On the other hand the Bureau could return to its definition of 1920 in which the population center is the point of minimum aggregate travel but calculate it correctly. While it may have been a daunting task back then, it is now a calculation that can be done easily with the data provided by the Census Bureau.

\subsection*{Calculating the hub}

On its web pages  \cite{PopCenters00} the Census Bureau provides 2000 census data for the 65,443 census tracts in the 50 states. Each line of the comma-delimited text file contains six numbers. The first three numbers identify the state, county, and census tract. The last three numbers give the population of the tract and the latitude and longitude in degrees of the population center of the tract. Note that the longitude is negative because of the convention that east is the positive direction from Greenwich. Here are a few lines from that file with the commas removed for readability.
\vspace{2mm}

\texttt{
\begin{tabular}{llllll}
State&County&Tract&Pop.&Latitude&Longitude\\  
\hline 
06&077&005404&6511&+37.732419&-121.425296\\
06&077&005500&6876&+37.71513&-121.322774\\
06&079&010000&6803&+35.701192&-120.801674\\
06&079&010100&8787&+35.634833&-120.69533\\
06&079&010201&4687&+35.642344&-120.655632\\
06&079&010202&4180&+35.606719&-120.650357\\
06&079&010203&8069&+35.615338&-120.670017 
\end{tabular} }

\vspace{3mm}
We wrote a program to use this data to compute the function that is the aggregate great circle distance from the entire population to the point with latitude $\phi$ and longitude $\lambda$. Then we minimized this function on a grid with one degree increments. Each evaluation of the function on a grid point took about two seconds, and so in several minutes all points with any chance of being optimal could be checked. We found the minimum to be $(39^\circ,87^\circ)$, a location in Greene County, Indiana, about 70 miles southwest of Indianapolis. From this point the average distance to any person in the country is 795 miles, while the average distance to Memphis is 843 miles. It is interesting to note that as FedEx has grown it has established some secondary hubs, one of them being in Indianapolis. Furthermore, the optimal location is only about 85 miles northwest of Louisville where UPS has established its main hub.
The optimal location is 275 miles from the Census Bureau's population center $(37.7^\circ,91.8^\circ)$ located in Phelps County, Missouri and 315 miles from Memphis $(35^\circ,90^\circ)$. 
\begin{figure}
\vspace{.3in}
\centerline {
\includegraphics[width=5in]{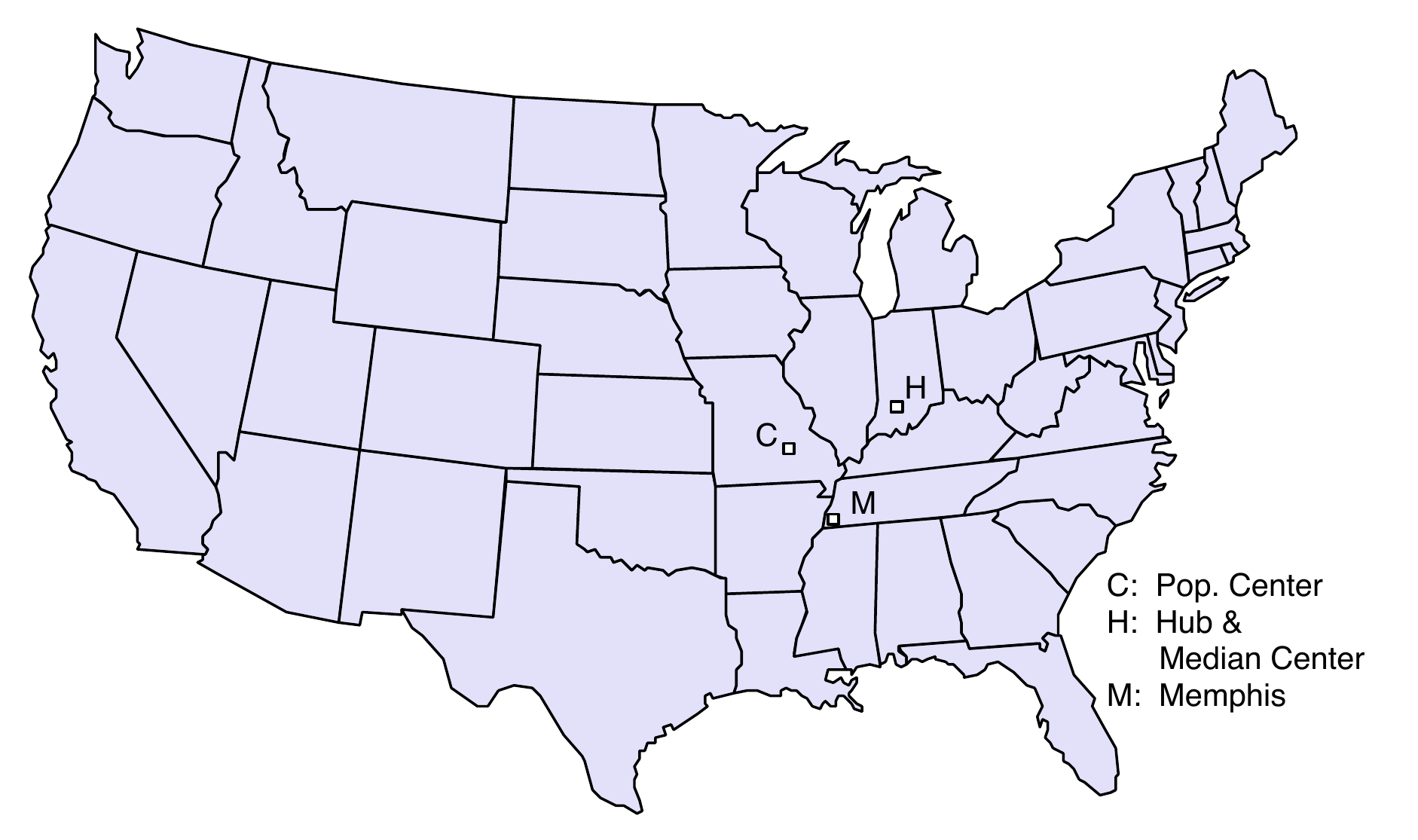}  
}
\caption{Location of the 2000 Census Bureau population center (C), the 2000 U.S. population hub (H), and the FedEx Memphis hub (M). The median center is very close to H.}
\end{figure}
The Census Bureau also computes a point called the ``median center of population'' by finding the median latitude and median longitude for the U.S. population; one half of the population lives north of the median latitude and one half of the population lives east of the median longitude. In the plane, this point minimizes the expected distance using the 1-norm to define distance between $x, y$ by $\| x - y \| = |x_1-y_1| + |x_2-y_2|$, and so the Census Bureau is essentially using the distance defined by the sum of differences in latitude and longitude. In general the optimal points depend on the metric used, but in this case the median center of population when rounded to the nearest degree is the same as the hub we found.  (Their precise answer is $(38.75644^\circ, 86.93074^\circ)$. See \cite{Census01}.) Refer to Figure 5 for a map showing these locations.

\section*{Generalizations and Extensions}
\subsection*{More than four points in the plane}
With five or more points in the plane there is no simple description of the hub. The difficulty stems from the fact that the set $\{(u_1,\ldots,u_k) \, | \, \|u_i\|=1,\sum u_i =0\} \subset (\R^2)^k$ has dimension $k-2$. The one-dimensional orthogonal group $O(2)$ acts on the set by rotations; the quotient space is the space of geometrically distinct configurations and has dimension $k-3$. For $k=3$ there is a unique configuration up to the action of $O(2)$, and for $k=4$ there is a one-dimensional set of geometrically distinct configurations parameterized by the smallest angle between two of the $u_i$. For $k \geq 5$ the dimension of the set of geometrically distinct configurations is two or more, and there are too many possibilities. Figure 6 is a typical example with $k=6$, showing the unit vectors and the associated hexagon that has no obvious symmetry. The hexagon can be deformed continuously into other hexagons with three degrees of freedom. (Challenge to the reader: identify the degrees of freedom.) 

\begin{figure}
\vspace{-.5in}
\centerline { 
\includegraphics[width=4.5in]{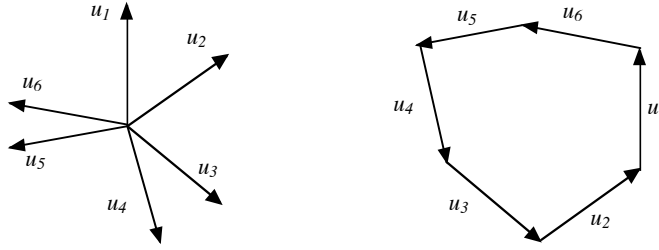}   
}
\vspace{-3cm}
\caption{A configuration of six unit vectors with zero sum.}
\end{figure}
\subsection*{From points to probability measures}
A finite number of points $\xset$ in $\R^n$ can be regarded as a discrete probability measure with each point having probability $1/k$. It is natural to define a hub for a probability measure $\rho$ as a point having minimal expected distance to points distributed according to $\rho$. Then Theorem 1 has the following generalization and can be proved in essentially in the same way. (An even more general result can be found in \cite{MilasevicDucharme87}.)

\begin{theorem}
Let $\rho$ be a probability measure on $\R^n$ with compact support not lying on a line. Then $\rho$ has a unique hub, and this hub is contained in the convex hull of the support. 
\end{theorem}

With this result we know that reasonable sets in $\R^n$ have unique hubs. For example, let $X$ be a bounded subset of $\R^n$ having positive Lebesgue measure $\lambda(X)$ and define $\rho$ to be normalized Lebesgue measure restricted to $X$. That is, for a measurable subset $E$
\[\rho (E)=\lambda(E \cap X)/\lambda(X).\]
There are only a few regions for which the hub can be exactly determined. For the interior of a rhombus or an ellipse the hub is the center, but for a region as simple as the interior of an isosceles triangle there does not appear to be an exact formula for the hub.

The theorem also applies to bounded curves in the plane for which $\rho$ is normalized arc length of the curve. Again there are very few curves for which the hub can be described exactly.

\subsection*{Hubs on a sphere}
Essential to the proof of uniqueness of a hub for points in $\R^n$ is that the function to be minimized is strictly convex, a property that depends strongly on the Euclidean structure of $\R^n$. Interesting and challenging questions immediately arise in other metric spaces such as the sphere. Given points $\xset$ on the sphere, 
a hub is a point $h$ minimizing the function $f(x)= \sum_i d(x_i,x)$, where $d(x,y)$ is the great circle distance between $x$ and $y$. Hubs \emph{exist} because the function to be minimized is continuous and the sphere is compact, but \emph{uniqueness} may fail.  An easy example is the case of two antipodal points, say the north and south poles; in this case any point on the equator is a hub. 

A subset of the sphere is spherically convex if it contains the geodesics between any two points in the subset, and the (spherical) convex hull of a subset is the smallest convex set containing the subset. Aly, Kay, and Litwhiler \cite{AlyKayLitwhiler79} prove that if $x_1,\ldots,x_k$ lie in an open hemisphere, then the hub (or hubs) must be in the convex hull of the points. One might conjecture in that case that the hub is unique, but even for three points that is not always true. For example, three points equally spaced on the same latitude just above the equator have the property that the points themselves are minima and they are the only minima. A complete description of the minima for three points was given by Cockayne in 1972 \cite{Cockayne72}. It would be interesting to find reasonable assumptions that guarantee unique hubs for $k$ points, as well as to prove existence and uniqueness results for more general probability measures on $S^2$ and on higher dimensional spheres.

\subsection*{Multiple hubs}

As FedEx and other package shipping companies have grown, they have established additional hubs so that packages from Boston to New York, for example, are shipped through an intermediate location on the East Coast rather than through Memphis. Deciding where to put a second or third hub is a ``multiple facility location'' problem in operations research, and there is ongoing interest in such problems. The two hub problem for finite sets in Euclidean space is the following. Given $\xset$ in $\R^n$ and two hubs $u, v \in \R^n$ a package shipped from $x_i$ to $x_j$ goes via the hub that results in the shorter total distance traveled by the package. 
One can prove the existence of minimizing pairs using continuity and compactness arguments, but uniqueness may not hold in the generality of the one hub case because the function $f$ is no longer convex. Explicitly finding solutions, however, appears to be virtually impossible. Consider, for example, the two hub problem for the uniform distribution on $[0,1]$. The function to be minimized is
\[ f(u,v)= \int_0^1 \int_0^1 \min(|x-u|+|y-u|,|x-v|+|y-v|)\, dx\,dy .\]
Numerical optimization gives the optimal locations as approximately $0.29$ and $0.71$. For these hubs the average distance a package travels is $0.39$, a significant decrease from the one hub average, which is $1/2$, and not too more than the average distance between any two points, which is $1/3$.


\begin{thebibliography}{99}


\bibitem{AlyKayLitwhiler79}
A.~A. Aly, D.~C. Kay, and D. W.~Litwhiler Jr., Location
  dominance on spherical surfaces, \emph{Oper. Res.} \textbf{27} (1979) 972--981.

\bibitem{Cockayne72}
E.~J. Cockayne, On {F}ermat's problem on the surface of a sphere, \emph{Math.
  Mag.} \textbf{45} (1972) 216--219. 

\bibitem{DreznerHamacher02}
Z. Drezner and H.W. Hamacher, eds., \emph{Facility Location: Application and Theory}, Springer, New York, 2002.

\bibitem{Editor30}
Editor's note on the center of population and point of minimum travel,
\emph{ J. Amer. Statist. Assoc.} \textbf{25} (1930)
  447--452.

\bibitem{Eells30}
W.~C. Eells, A mistaken conception of the center of population,\emph{ J. Amer.
  Statist. Assoc.} \textbf{25} (1930) 33--40.
  
\bibitem{FedEx09} FedEx Corp., FedEx history (2009); available at \texttt{http://about.fedex.designcdt.
com/our\_company/company\_information/fedex\_history}.

\bibitem{Haldane48}
J.~B.~S. Haldane, Note on the median of a multivariate distribution,
  \emph{Biometrika} \textbf{35} (1948) 414--415.
  
\bibitem{Hajja94}
M.~Hajja, An advanced calculus approach to finding the Fermat point, \emph{Math. Mag.} \textbf{67} (1994) 29--34. 

\bibitem{Honsberger73}
R.~Honsberger, \emph{Mathematical Gems I}, Mathematical Association of America, Washington DC, 1973.


\bibitem{MilasevicDucharme87}
P.~Milasevic and G.~R. Ducharme, Uniqueness of the spatial median, \emph{Ann.
  Statist.} \textbf{15} (1987) 1332--1333. 

\bibitem{PSU93}
 A.~L.~Peressini, F.~E.~Sullivan, and J.~J.~Uhl, Jr., \emph{The Mathematics of Non-Linear Programming}, Springer, New York, 1993.

\bibitem{PopCenters00}
U.S. Census Bureau, 
Centers of population for census 2000, index of maps, documentation, and tables (2000); available at \\ \texttt{http://www.census.gov/geo/www/cenpop/cntpop2k.html}.
  
\bibitem{Census01}
---------,
Centers of population computation for 1950, 1960, 1970, 1980, 1990 and 2000, 
Tech. report, Washington DC, 2001; available at \\ \texttt{ http://www.census.gov/geo/www/cenpop/cntpop2k.html}.

\bibitem{Wesolowsky93}
G.~O. Wesolowsky, The {W}eber problem: history and perspectives,
  \emph{Location Science} \textbf{1} (1993) 5--23.

\end{thebibliography}
\end{document}